\documentclass[12pt,reqno]{article}
\usepackage{color}
\usepackage{amsmath,amssymb}
\usepackage{amsthm}
\usepackage{fullpage}
\usepackage{epsf}
\usepackage[colorlinks=true,linkcolor=webgreen,filecolor=webbrown,citecolor=webgreen]{hyperref}

\definecolor{webgreen}{rgb}{0,.5,0}
\definecolor{webbrown}{rgb}{.6,0,0}

\DeclareMathOperator{\red}{red}
\DeclareMathOperator{\canonize}{canonize}

\newtheorem{thm}{Theorem}
\newtheorem{lem}{Lemma}

\newcommand{\seqnum}[1]{\href{http://oeis.org/#1}{\underline{#1}}}

\begin{document}


\begin{center}
\vskip 1cm{\LARGE\bf Avoiding colored partitions of two elements \\in the pattern sense}

\vskip 1cm
\large
Adam M.Goyt\\
Department of Mathematics\\
Minnesota State University Moorhead\\
Moorhead, MN 56563, USA \\
\href{mailto:goytadam@mnstate.edu}{\tt goytadam@mnstate.edu}\\
\ \\
Lara K. Pudwell \\
Department of Mathematics and Computer Science\\
Valparaiso University\\
Valparaiso, IN 46383, USA\\
\href{mailto:Lara.Pudwell@valpo.edu}{\tt Lara.Pudwell@valpo.edu}\\
\end{center}

\begin{abstract}
Enumeration of pattern-avoiding objects is an active area of study with connections to such disparate regions of mathematics as Schubert varieties and stack-sortable sequences.  Recent research in this area has brought attention to colored permutations and colored set partitions. A colored partition of a set $S$ is a partition of $S$ with each element receiving a color from the set $[k]=\{1,2,\dots,k\}$.  Let $\Pi_n\wr C_k$ be the set of partitions of $[n]$ with colors from $[k]$. 

In an earlier work, the authors study pattern avoidance in colored set partitions in the equality sense.  Here we study pattern avoidance in colored partitions in the pattern sense.  We say that $\sigma\in\Pi_n\wr C_k$ contains $\pi\in \Pi_m\wr C_\ell$ in the pattern sense if $\sigma$ contains a copy $\pi$ when the colors are ignored and the colors on this copy of $\pi$ are order isomorphic to the colors on $\pi$. Otherwise we say that $\sigma$ avoids $\pi$.

We focus on patterns from $\Pi_2\wr C_2$ and find that many familiar and some new integer sequences appear.  We provide bijective proofs wherever possible, and we provide formulas for computing those sequences that are new.     

\end{abstract}

\section{Introduction}

Knuth \cite{K73} first introduced pattern avoidance in permutations, which continues to be an active area of research today.  Given a string of integers $s$, the {\it reduction} of $s$, denoted $\red(s)$ is the unique string obtained by replacing the $i$th smallest integer(s) of $s$ with $i$; we say that $\red(s)$ is \emph{order-isomorphic} to $s$.  For example, $\red(18494) = 13242$.  Now, let $\mathcal{S}_n$ denote the set of permutations of length $n$, and consider $q \in \mathcal{S}_n$ and $p \in \mathcal{S}_m$.  We say $q$ \emph{contains} $p$ as a pattern if there exist indices $1 \leq i_1 < i_2 < \cdots < i_{m-1} < i_m \leq n$ such that $\red(q_{i_1}q_{i_2}\cdots q_{i_m})=p$.  Otherwise, we say $q$ \emph{avoids} $p$.  Further, let $\mathcal{S}_n(p)$ denote the set of permutations of length $n$ that avoid $p$, and let $\text{s}_n(p) = \left|\mathcal{S}_n(p)\right|$.  It is straightforward to see that $\text{s}_n(12) = 1$ for $n \geq 0$ because the only permutation of length $n$ that avoids 12 is the decreasing permutation.  It is also well-known that given any permutation $p \in \mathcal{S}_3$, $\text{s}_n(p)=C_n$ where $C_n = \dfrac{\binom{2n}{n}}{n+1}$ is the $n$th Catalan number \cite{SS85}.

Pattern avoidance has been studied in contexts other than permutations.  In particular, Klazar \cite{K96} introduced the notion of pattern-avoidance in set partitions, and Klazar, Sagan, and Goyt \cite{G08, GS09, K00, K00b, Sagan10} did further work.  More recently the current authors introduced the set of $k$-colored set partitions of $[n]$, denoted $\Pi_n \wr C_k$, in an enumerative context~\cite{GPTBA1}.  The authors define  three distinct definitions of a colored partition pattern: EQ, LT, and pattern.  Although previous work focused on avoiding patterns in the EQ sense, in this paper we consider avoiding one or more partition patterns in the pattern sense.  The enumeration of such colored set partitions turns out to be related to a number of other combinatorial objects.  As in the previous paper, we will primarily focus on $\Pi_n \wr C_2$, that is partitions with only 2 colors. Further generalization is certainly possible.

A {\it partition} $\sigma$ of the set $S \subseteq \mathbb{Z}$, written $\sigma \vdash S$, is a family of nonempty, pairwise disjoint subsets $B_1, B_2, \dots , B_k$ of $S$ called \emph{blocks} such that $\displaystyle{\bigcup_{i=1}^k B_i = S}$. We write $\sigma = B_1/B_2/\dots/B_k$.  Note that because $B_1, \dots , B_k$ are sets, order of elements within a block does not matter.  We write elements of a block in increasing order, and we write the blocks in the canonical order where $$\min(B_1)<\min(B_2)<\cdots<\min(B_k).$$  For example, $157/238/4/6 \vdash [8]$.  We will study partitions of $[n]=\{1,\dots,n\}$, so we define $\Pi_n = \{\pi\text{ }|\text{ }\pi \vdash [n]\}$.

Given this canonical ordering, we may associate with any $\pi \in \Pi_n$ a \emph{canonical word} $w$ of length $n$, such that $w_i=j$ if and only if element $i$ is a member of $B_j$.  For example, the canonical word associated with $157/238/4/6$ is $12231412$.  If we consider a set partition $\sigma$ of a subset of $[n]$, the \emph{canonization} of $\sigma$, denoted $\canonize(\sigma)$, is the set partition obtained by replacing the $i$th smallest element of $\sigma$ with $i$ and writing it in the canonical order described above.  For example, $\canonize(17/6/8)=13/2/4$.  One may similarly canonize any subword of a canonical word since the subword corresponds to a subpartition.  For example, $\canonize(2312) = 1231$ since $2312$ corresponds to a set partition where the first and fourth smallest elements are in block 2, the second smallest element is in block 3, and the third smallest element is in block 4.  After canonizing this partition, we have $14/2/3$, whose canonical word is indeed $1231$.

Now, suppose $\sigma \in \Pi_n$ and $\pi \in \Pi_m$, we say that $\sigma$ {\it contains} $\pi$ as a partition pattern if and only if there exist $m$ elements of $\sigma$ whose corresponding canonical word is the canonical word for $\pi$.  Otherwise, we say that $\sigma$ {\it avoids} $\pi$.  For example, $\sigma=167/238/4/5$ contains $\pi_1=15/2/34$ as evidenced by $\canonize(67/28/5)=\canonize(28/5/67)=15/2/34$, but $\sigma$ avoids $\pi_2 = 123/4/5$ since for neither block of size 3 are there two larger elements in other blocks.

We are particularly interested in the set of \emph{colored partitions} of $[n]$ with $k$ colors, written $\Pi_n \wr C_k$, which is the set of partitions $\sigma$ such that $\sigma \vdash [n]$, and each element $[n]$ is assigned a color from the set $[k]$. We write these colors as superscripts on each element of $\sigma$, so for example $1^13^2/2^2/4^1 \in \Pi_4 \wr C_2$.  The wreath product notation $\Pi_n \wr C_k$ was originally adopted to emphasize the parallel between colored partitions and the set of colored permutations $S_n \wr C_k$ which is a wreath product in the true algebraic sense.  We say that $\sigma \in \Pi_n \wr C_k$ contains $\pi \in \Pi_n \wr C_k$ in the pattern sense if and only if the uncolored version of $\sigma$ contains an uncolored copy of $\pi$ whose colors are order-isomorphic to the colors on $\pi$.  For example, $1^13^1/2^2/4^2$ contains the patterns $1^12^1$, $1^1/2^2$, and $1^2/2^1$ but not $1^12^2$.  This is distinct from pattern avoidance in the EQ sense where the colors on the copy of $\pi$ must equal the colors of $\pi$, and the LT sense where the colors must be less than or equal to the colors of $\pi$.  For convenience, in most of the paper, partitions will be written as colored canonized words rather than in block form, so, for example, $1^13^1/2^2/4^2$ will be written as $1^12^21^13^2$.  We write $\Pi^{pat}_n \wr C_k (S)$ for the set of partitions in $\Pi_n \wr C_k$ which avoid all patterns in the set $S$ in the pattern sense, and $\Pi^{eq}_n \wr C_k (S)$ for those that avoid all patterns of $S$ in the EQ sense.  In addition to the enumeration problem of computing $\left|\Pi^{pat}_n \wr C_k (S)\right|$ for fixed $k$ and $S$, we are interested in classifying when $\left|\Pi^{pat}_n \wr C_k (S_1)\right| = \left|\Pi^{pat}_n \wr C_k(S_2)\right|$ for two distinct sets of patterns $S_1$ and $S_2$.  Two such sets of patterns are said to be \emph{Wilf-equivalent}.

In this paper we enumerate members of $\Pi_n \wr C_2$ that avoid any set of colored partitions of $[2]$, and we completely describe all Wilf-equivalences arising from these enumerations.  In Sections \ref{S:one}, \ref{S:two}, \ref{S:three}, and \ref{S:four}, we consider partitions that avoid 1, 2, 3, 4, or 5 patterns in turn.  Then, in Section \ref{S:generalizations} we provide some further generalizations.  

\section{Avoiding one pattern} \label{S:one}

We first note that there are $B(n)\cdot2^n$ partitions in $\Pi_n\wr C_2$, where $B(n)$ is the $n$th Bell number.  These can be thought of the colored partitions which avoid no patterns.

By the above there are eight $2$-colored partitions of $[2]$; namely, $1^11^1, 1^11^2, 1^21^1, 1^21^2, 1^12^1,\\ 1^12^2, 1^22^1, 1^22^2$.  Since we are working with partitions in the pattern sense, we treat the colors as permutations and make a couple observations.  First, note that because $\red(22)=11$, $a^2b^2$ is trivially equivalent to $a^1b^1$, so it is omitted for any choices of $a$ and $b$.  Further if $q=q_1\cdots q_n \in [k]^n$, then the \emph{reversal} of $q$, denoted $q^r$, is $q_n \cdots q_1$ and the \emph{complement} of $q$, denoted $q^c$ is $(k+1-q_1)\cdots (k+1-q_n)$.  Our work is further reduced with the following lemma:

\begin{lem}
Consider $\pi=\pi_1^{c_1}\pi_2^{c_2}\cdots \pi_n^{c_n}$ and $\rho=\rho_1^{d_1}\rho_2^{d_2}\cdots \rho_n^{d_n}\in \Pi_n\wr C_k$.  Let $C=c_1\cdots c_n$ and $D=d_1 \cdots d_n$.  If $\pi_1 \cdots \pi_n = \rho_1 \cdots \rho_n$ and $C =D^r$ or $C=D^c$, then $\pi$ and $\rho$ are Wilf-equivalent.
\label{L:lemma}
\end{lem}

Lemma \ref{L:lemma} follows naturally from the action of the dihedral group $D_4$ on permutations.  Further, this lemma tells us that $1^11^2$ is Wilf-equivalent to $1^21^1$ and $1^12^2$ is Wilf-equivalent to $1^22^1$.  Thus, there are only 4 such partition patterns to consider in the pattern sense.  It turns out that each of these patterns was already considered (possibly indirectly) in the literature.

\begin{itemize}
\item $\Pi^{pat}_n \wr C_2(1^11^1)$is equivalent to $\Pi^{eq}_n \wr C_2(1^11^1,1^21^2)$.  These partitions are in bijection with the involutions of $[2n]$ that are invariant under the reverse-complement map \cite[Theorem 3.5]{GPTBA1}.  Such partitions are counted by OEIS sequence \seqnum{A000898}.
\item $\Pi^{pat}_n \wr C_2(1^11^2)$ is equivalent to $\Pi^{eq}_n \wr C_2(1^11^2)$.  The corresponding integer sequence is \seqnum{A209801} \cite[Theorems 2.1, 2.2]{GPTBA1}.
\item $\Pi^{pat}_n \wr C_2(1^12^1)$ is equivalent to $\Pi^{eq}_n \wr C_2(1^12^1,1^22^2)$.  Such partitions are in bijection with the non-empty proper subsets of an $(n+1)$-element set~\cite[Theorem 3.3]{GPTBA1}. These partitions are counted by OEIS sequence \seqnum{A000918}.
\item $\Pi^{pat}_n \wr C_2(1^12^2)$ is equivalent to $\Pi^{eq}_n \wr C_2(1^12^2)$.  The corresponding integer sequence is \seqnum{A011965} \cite[Theorems 2.3, 2.4]{GPTBA1}. 
\end{itemize}

Now that we have enumerated colored partitions that avoid one element of $\Pi_2\wr C_2$, we consider partitions that avoid more than one element of $\Pi_2\wr C_2$.  Note that in the sequel, we are only concerned with pattern-type avoidance, so we write $\Pi_n \wr C_k(S)$ in lieu of $\Pi^{pat}_n \wr C_k(S)$.

\section{Avoiding 2 patterns}\label{S:two}

There are 8 distinct Wilf classes of patterns avoiding a pair of elements of $\Pi_2\wr C_2$.  Table \ref{T:twodata} presents each of these Wilf classes along with the first 6 terms of $\left| \Pi_n \wr C_2(S)\right|$, and the appropriate sequence entry from the Online Encyclopedia of Integer Sequences~\cite{OEIS}.  Pattern sets that are known to be equivalent via Lemma \ref{L:lemma} are given on the same line.  We address each of these classes in turn.

\begin{center}
\begin{table}[hb]
\begin{center}
\begin{tabular}{|c|c|c|c|}
\hline
Class&Patterns&Sequence&OEIS number\\
\hline
1&$\{1^11^1,1^12^1\}$&$2, 4, 0, 0, 0, 0$&Trivial\\
\hline
2&$\{1^11^2,1^12^1\}\equiv\{1^21^1,1^12^1\}$&$2, 5, 10, 19, 36, 69$&\seqnum{A052944}\\
&$\{1^12^1,1^12^2\}\equiv\{1^12^1,1^22^1\}$&&\\
\hline
3&$\{1^11^1,1^12^2\}\equiv\{1^11^1,1^22^1\}$&$2, 5, 10, 21, 46, 107$&\seqnum{A208275}\\
\hline
4&$\{1^11^1,1^11^2\}\equiv\{1^11^1,1^21^1\}$&$2, 5, 14, 43, 142, 499$&\seqnum{A005425}\\
\hline
5&$\{1^12^2,1^22^1\}$&$2, 6, 16, 44, 134, 468$& \seqnum{A209629}\\
\hline
6&$\{1^11^2,1^22^1\}\equiv\{1^21^1,1^12^2\}$&$2, 6, 18, 56, 188, 695$&\seqnum{A209797}\\
\hline
7&$\{1^11^2,1^12^2\}\equiv\{1^21^1,1^22^1\}$&$2, 6, 20, 75, 312, 1421$&\seqnum{A052889}\\
\hline
8&$\{1^11^2,1^21^1\}$&$2, 6, 22, 94, 454, 2430$&\seqnum{A001861}\\
\hline
\end{tabular}
\end{center}
\caption{Enumeration data for 2-colored partitions avoiding a pair of partition patterns}
\label{T:twodata}
\end{table}
\end{center}

Note that Lemma \ref{L:lemma} explains all Wilf equivalences except in Class 2.  By Lemma \ref{L:lemma} the first two pattern sets in Class 2 are equivalent and the last two pattern sets in Class 2 are equivalent. To show that all four are equivalent we will show that $|\Pi_n\wr C_2(1^21^1,1^12^1)|=|\Pi_n\wr C_2(1^12^1,1^22^1)|$.  This can be done by understanding the structure of the partitions in these two sets and describing a bijection between them.  

First we will consider the partitions in $\Pi_n\wr C_2(1^12^1,1^22^1)$.  Since we have only two colors and avoid $1^12^1$, there can be at most two blocks in this partition.  If there is exactly one block each element can have either color since there is no way to produce a copy of $1^22^1$ in a single block.  If there are two blocks, to avoid $1^12^1$, all elements in one block are colored 1 and all elements in the other block are colored 2.  Furthermore, to avoid $1^22^1$ the block colored 1 must contain the first $i$ elements and the block colored 2 must contain the last $n-i$ elements for $1\leq i\leq n-1$.  

Partitions in $\Pi_n\wr C_2(1^21^1, 1^12^1)$ must also have at most two blocks.  If there is exactly one block then to avoid $1^21^1$ the first $i$ elements must be colored 1 and the last $n-i$ elements must be colored 2 where $0\leq i\leq n$.  On the other hand if there are two blocks then to avoid $1^12^1$ all elements in one block are colored 1 and all elements in the other block are colored 2.  Each element may be in either block.

The bijection is as follows.  Given $\pi\in\Pi_n\wr C_2(1^12^1,1^22^1)$, if $\pi$ has two blocks, put everything together in one block.  If $\pi$ has one block separate the elements into two blocks with all of the 1 colored elements in one block and all of the two colored elements in one block.  Note that the bijection is the identity on those partitions where every element has the same color.  This is clearly invertible and well-defined based on the descriptions above.

Now that all Wilf equivalences of pairs of patterns have been addressed, we consider the enumeration of each class.

\subsection {Class 1}

\begin{thm}  $$\left|\Pi_n\wr C_2(1^11^1,1^12^1)\right|=\begin{cases}
2&\text{if }n=1,\\
4&\text{if }n=2,\\
0&otherwise.\\
\end{cases}$$\end{thm}

\begin{proof}
If a colored partition avoids $1^11^1$ then no two elements in the same block may have the same color.  Further if it avoids $1^12^1$, then no two elements in different blocks may have the same color.  Thus, if $\sigma \in \Pi_n \wr C_2(1^11^1,1^12^1)$ then each element of $\sigma$ must have its own color.  There are two such partitions of $[1]$ and four such partitions of $[2]$, but there are no elements of $\Pi_n \wr C_2(1^11^1,1^12^1)$ for $n \geq 3$ since there are only 2 available colors and more than 3 distinct elements in our set partitions.\end{proof}

\subsection{Class 2}

\begin{thm}  For $n\geq1$, $$\left|\Pi_n\wr C_2(1^12^1,1^22^1)\right|=2^n+n-1.$$\end{thm}

\begin{proof}  By the analysis above, either $\sigma$ has one block and the elements are colored arbitrarily, which gives $2^n$ possibilities, or it has a nonempty block of 1's each colored 1 followed by a nonempty block of 2s each colored 2, which gives $n-1$ possibilities.
\end{proof}

\subsection{Class 3}

\begin{thm} For $n \geq 1$,

\begin{align*}
\left|\Pi_n\wr C_2(1^11^1,1^12^2)\right|&=\sum_{j\geq 0}\sum_{i=1}^n \binom{i-1}{j}\binom{n-i}{j}j! +\sum_{j\geq 0} \sum_{i=2}^n(i-1)\binom{i-2}{j}\binom{n-i}{j}j!\\
&+ \sum_{j\geq 0}\sum_{i=1}^{n-1}\binom{i-1}{j}\binom{n-i-1}{j}j!+1.
\end{align*}
 \end{thm}

\begin{proof}
Note that if $\sigma \in \Pi_n\wr C_2(1^11^1,1^12^2)$, then any block of $\sigma$ has at most two elements because of the $1^11^1$ restriction.  Further the elements of a two-element block must have different colors.  Let $i$ be the smallest element of $\sigma$ with color 1.  We break our argument into 4 cases.

Case 1: Suppose $i$ is in a block of size 1.  By definition, all numbers less than $i$ have color 2.  Because of the $1^12^2$ restriction, all elements greater than $i$ have color 1.  We may create $j$ blocks of size 2 by pairing off $j$ elements less than $i$ with $j$ elements greater than $i$ in $\binom{i-1}{j}\binom{n-i}{j}j!$ ways.  Summing over all reasonable values of $i$ and $j$ gives the first term in the formula above.

Case 2: Suppose $i$ is in a block with an element less than itself.  By definition all numbers less than $i$ have color 2, and because of the $1^12^2$ restriction, all elements greater than $i$ have color 1.  There are $i-1$ possible elements to pair with $i$.  Then we may create $j$ blocks of size 2 by pairing off $j$ elements less than $i$ with $j$ elements greater than $i$ in $\binom{i-2}{j}\binom{n-i}{j}j!$ ways.  Summing over all reasonable values of $i$ and $j$ gives the second term in the formula above.

Case 3: Suppose $i$ is in a block with an element $k$ larger than itself.  Because of the $1^12^2$ restriction all numbers greater than $i$ and in a different block from $i$ must have color 1.  On the other hand, because of the same restriction, all numbers less than $k$ and in a different block from $k$ must have color 2.  This means that unless $k=i+1$, there are no such partitions.  Thus we have $i$ in the same block as $i+1$, all elements less than $i$ have color 2 and all elements greater than $i+1$ have color 1.  Now we can again create $j$ blocks of size 2 by pairing off $j$ elements less than $i$ with $j$ elements greater than $i+1$ in $\binom{i-1}{j}\binom{n-i-1}{j}j!$ ways.  Summing over all reasonable values of $i$ and $j$ gives the third term in the formula above.

Case 4: $i$ does not exist because all elements have color 2.  In this case, our partition must consist only of blocks of size 1.  This accounts for the 1 at the end of the formula.
\end{proof}

\subsection{Class 4}

\begin{thm}
Define $a_n=\left|\Pi_n\wr C_2(1^11^1,1^11^2)\right|$.  Then for $n \geq 3$, $$a_n=2a_{n-1}+(n-1)a_{n-2}.$$
\end{thm}

\begin{proof}
Let $\sigma\in\Pi_n\wr C_2(1^11^1,1^11^2)$.  Since $\sigma$ avoids $1^11^1$ there can be no more than two elements in a block.  

Consider the element $n$.  If $n$ is in a block of size 1, there are two ways to choose a color for $n$ and then $a_{n-1}$ ways to partition and color the remaining elements.

If $n$ is in a block with another element, then it must be the case that the smaller element has color $2$ and $n$ has color $1$.  There are $n-1$ choices for which element to pair with $n$ and $a_{n-2}$ ways to partition and color the remaining elements.
\end{proof}

We have that $a_1=2$ and $a_2=5$ by brute force enumeration.  Standard algebra shows that this sequence has exponential generating function $\displaystyle{e^{\left(2x + \frac{x^2}{2}\right)}}$.

We also note that $\Pi_n\wr C_2(1^11^1,1^11^2)$ has a nice relationship to a particular set of pattern-avoiding permutations.  In particular, $\mathcal{S}_{n+1}\left(12\!-\!3,214\!-\!3\right)$ is the set of permutations of length $n+1$ avoiding the vincular permutation patterns $12\!-\!3$ and $214\!-\!3$.  Here a copy of $12\!-\!3$ is a copy a 123 pattern where the roles of ``1'' and ``2'' must be played by adjacent permutation elements, and similarly, a copy of $214\!-\!3$ is a copy of a 2143 pattern where the roles of `2'', ``1'', and ``4'' are played by adjacent elements.  We have the following:

\begin{thm}
For $n \geq 0$,
$\left|\Pi_n\wr C_2(1^11^1,1^11^2)\right| = \left|\mathcal{S}_{n+1}\left(12\!-\!3,214\!-\!3\right)\right|$
\label{T:class34}
\end{thm}

This theorem will be proved via a bijection. We first make some preliminary observations about the structure of the partitions and the structure of the permutations.

If a 2-colored set partition avoids $1^11^1$ and $1^11^2$ we see that each block either consists of a single element, or it consists of a pair of elements $i<j$ where $i$ has color 2 and $j$ has color $1$.

If a permutation $q$ of length $n+1$ avoids $12\!-\!3$ and $214\!-\!3$, then all elements before $n+1$ appear in decreasing order (lest we create a $12\!-\!3$ pattern).  Further, if $n+1$ is in position $j \geq 3$, then the first $j-2$ elements must precisely be $(n)(n-1)(n-2)\cdots((n+1)-(j-2))$, otherwise $q_{j-2}q_{j-1}(n+1)$ followed by the missing element of this initial decreasing run form a $214\!-\!3$ pattern. Thus we see a typical $\{12\!-\!3, 214\!-\!3\}$-avoiding permutation has the form $n(n-1)(n-2)\cdots(n-j+3)k(n+1)q^*$ where $1 \leq k \leq n-j+2$ and $q^*$ is a permutation of $[n-j+2]\setminus \{k\}$ that avoids $12\!-\!3$ and $214\!-\!3$.  In particular if we know that $i$ letters appear after the letter $n+1$, then we have $i+1$ choices for what letter immediately precedes $n+1$, and we must arrange the final $i$ elements so that they avoid $12\!-\!3$ and $214\!-\!3$.  Otherwise, the structure of the permutation is predetermined.

\begin{proof}[Proof of Theorem \ref{T:class34}]

We present a bijection $f: \Pi_n \wr C_2(1^11^1,1^11^2) \to \mathcal{S}_{n+1}(12\!-\!3,214\!-\!3)$.

Our map is recursive, so we define a few base cases.  Let $f(\emptyset)=1$, $f(1^1)=21$ and $f(1^2)=12$.

Now, consider partition $\sigma \in \Pi_n \wr C_2(1^11^1,1^11^2)$.  Note that from the proof of Theorem \ref{T:class34}, the element $n$ may be in a block with another element, and if this is the case there is one possible coloring, or $n$ may be in a block by itself colored in one of two ways.

\begin{itemize}
\item If $n$ is in a block of size 2 with element $j$, then let $\sigma^{\prime}$ be the canonized partition formed by deleting this block.  Then $f(\sigma)=j(n+1)f^*(\sigma^{\prime})$ where $f^*(\sigma^{\prime})$ is the word order isomorphic to $f(\sigma^{\prime})$ on the alphabet $[n]\setminus\{j\}$.
\item If $n$ is in a block of size 1 with color 1, then let $\sigma^{\prime}$ be the partition formed by deleting $n$.  Then $f(\sigma)=(n+1)f(\sigma^{\prime})$.
\item If $n$ is in a block of size 1 with color 2, then let $\sigma^{\prime}$ be the partition formed by deleting $n$.  Note that $f(\sigma^{\prime})$ is a permutation of $[n]$.  Let $\hat{q}$ be the part of $f(\sigma^{\prime})$ appearing before $n$ and let $\tilde{q}$ be the part of $f(\sigma^{\prime})$ appearing after $n$.  Then $f(\sigma)=(n)\hat{q} (n+1) \tilde{q}$.
\end{itemize}

Note that this process has a clear inverse.  In particular, consider a permutation $q \in \mathcal{S}_{n+1}(12\!-\!3,214\!-\!3)$.

We have $f^{-1}(1)=\emptyset$, $f^{-1}(12)=1^2$, and $f^{-1}(21)=1^1$.

\begin{itemize}
\item If $q_1=n+1$, then let $q^{\prime}=q_2 \cdots q_{n+1}$.  We have $f^{-1}(q)=f^{-1}(q^{\prime})/n^1$.
\item If $q_2=n+1$ and $q_1\not=n$, then let $q^{\prime}=q_3 \cdots q_{n+1}$, and let $j=q_1$.  We have $f^{-1}(q)=f^{-1*}(q^{\prime})/j^2n^1$, where $f^{-1*}(q^{\prime})$ canonizes to $f^{-1}(q^{\prime})$ but uses the integers of $[n-1]\setminus \{j\}$.
\item If $q_j=n+1$ where $j>2$ or $q_2=n+1$ and $q_1=n$, then let $\hat{q}=q_2\cdots q_{j-1}$ and $\tilde{q}=q_{j+1}\cdots q_{n+1}$.  We have $f^{-1}(q)=f^{-1}(\hat{q}(n)\tilde{q})/n^2$.
\end{itemize}
\end{proof}

For clarity, we provide an example of applying $f$ and its inverse.  Consider the partition $1^24^1/2^1/3^26^1/5^1/7^2$.

\begin{itemize}
\item Since $7$ is a block of size 1 with color 2, we first compute $f(1^24^1/2^1/3^26^1/5^1)$.
\begin{itemize}
\item Since $6$ is in a block with 3, $f(1^24^1/2^1/3^26^1/5^1)=37f(\canonize(1^24^1/2^1/5^1))=37f(1^23^1/2^1/4^1)$.
\begin{itemize}
\item Since $4$ is a block of size 1 with color 1, we have $f(1^23^1/2^1/4^1)=5f(1^23^1/2^1)$.
\begin{itemize}
\item Since $3$ is in a block with 1, we have $f(1^23^1/2^1)=14f^*(\canonize(2^1))=14f^*(1^1)$.
\item $f(1^1)=21$ by definition.
\item $f(1^23^1/2^1)=1432$.
\end{itemize}
\item $f(1^23^1/2^1/4^1)=5f(1^23^1/2^1)=51432$.
\end{itemize}
\item $f(1^24^1/2^1/3^26^1/5^1)=3761542$
\end{itemize}
\item Now $\hat{q}=3$ and $\tilde{q}=61542$, so $f(1^24^1/2^1/3^26^1/5^1/7^2)=73861542$.
\end{itemize}

Next, we apply the inverse map to $q=73861542$.

\begin{itemize}
\item Since $q_3=8$, compute $\hat{q}=3$, $\tilde{q}=61542$.  We have $f^{-1}(q)=f^{-1}(3761542)/7^2$
\begin{itemize}
\item Since $7$ is in the second position, we have $f^{-1}(3761542)=\\f^{-1*}(\red(61542))/3^26^1=f^{-1*}(51432)/3^26^1$.
\begin{itemize}
\item Since $5$ is in the first position, we have $f^{-1}(51432)=f^{-1}(1432)/4^1$.
\begin{itemize}
\item Since $4$ is in the second position, $f^{-1}(1432)=\\f^{-1*}(\red(32))/1^23^1=f^{-1*}(21))/1^23^1$.
\item $f^{-1}(21)=1^1$, by definition.
\item $f^{-1}(1432)=2^1/1^23^1$
\end{itemize}
\item $f^{-1}(51432)=2^1/1^23^1/4^1$
\end{itemize}
\item $f^{-1}(3761542)=2^1/1^24^1/5^1/3^26^1$
\end{itemize}
\item $f^{-1}(3761542)=2^1/1^24^1/5^1/3^26^1/7^2=1^24^1/2^1/3^26^1/5^1/7^2$, as desired.
\end{itemize}

\subsection{Class 5}

\begin{thm}
For $n \geq 1$, $\left|\Pi_n\wr C_2(1^12^2, 1^22^1)\right|=2^n + 2(B(n)-1)$, where $B(n)$ is the $n$th Bell number.
\end{thm}

\begin{proof}
If a colored partition avoids $1^12^2$ and $1^22^1$ any two elements in different blocks must have the same color.  We consider two cases.  If there is only one block, we may color all $n$ elements in $2^n$ ways.  On the other hand if there is more than one block, all elements must receive the same color.  There are 2 choices for a color and $(B(n)-1)$ ways to partition $[n]$ into more than one block.  This avoidance sequence is given by $2^n + 2(B(n)-1)$.\end{proof}

\subsection{Class 6}

\begin{thm} 
For $n\geq2$, 
\begin{align*}
\left|\Pi_n\wr C_2(1^11^2,1^22^1)\right|&=2B(n)+\sum_{j=2}^n\sum_{k=0}^{n-j}B(n-j-k)+B(n-1)\\
&+\sum_{j=2}^{n-1}B(j-1)\left(B(n-j)+\sum_{k=1}^{n-j}\left(\left(k+\binom{n-j}{k}\right) B(n-j-k)\right)\right).
\end{align*}
\end{thm}

\begin{proof}  Suppose that $\sigma\in \Pi_n\wr C_2(1^11^2,1^22^1)$.  The case where all elements of $\sigma$ have the same color is covered by the first term above.  

Now, suppose that each color appears at least once.  Let $j$ be the first element colored 1 and suppose that $j\geq2$.  Then every element in $[j-1]$ must be colored 2 and must appear in the same block as $j$.  We now choose $k$ elements from the remaining $n-j$ elements to appear in the block with $j$.  Each of these $k$ elements must be colored 1 to avoid a copy of $1^11^2$.  Also, each remaining element not in the block with $j$ must be colored 2 to avoid a copy of $1^22^1$.  Now, to avoid a copy of $1^22^1$ it must be precisely the elements $[j+k]$ in the first block.  This gives us the second term above.

The final case is the case where the first element is colored 1 and both colors appear at least once.  If $n$ is the first element colored 2 then the first $n-1$ elements must all be colored 1 and may be partitioned in any way.  Note that $n$ may not appear in a block with any of the first $n-1$ elements otherwise we would have a copy of $1^11^2$.  This gives the third term above.

Suppose, now that $j\not=n$ is the first element colored 2.  As above, the elements less than $j$ must all be colored $1$ and $j$ may not appear in a block with any of the elements preceding $j$.  If an element following $j$ appeared in a block with an element preceding $j$ then that element is colored 1 or 2.  If it is colored 1, it creates a copy of $1^22^1$ with $j$.  If it is colored 2 then it creates a copy of $1^11^2$ with an element appearing before $j$.  Thus, none of the elements following $j$ may appear in a block with an element preceding $j$.  We may partition these $j-1$ elements any way we like.

We now turn to the $n-j$ elements following $j$.  None of these elements could appear in the block with $j$ in which case there are $B(n-j)$ ways to partition them, and they must all be colored 2 to avoid $1^22^1$. 

Suppose there are $k\geq1$ of the elements $\{j+1,\dots,n\}$ in the same block as $j$.  If one of them is colored 1 then to avoid $1^22^1$ these $k$ elements must be $j+1,j+2,\dots,j+k$.  Now, there are $k$ ways to pick the smallest of the elements $j+1,\dots,j+k$ to be colored 1, and at this point the colors on these elements are determined.  If none of these $k$ elements is colored 1 then we may choose them in $\binom{n-j}{k}$ ways.  Once we have added these $k$ elements to the block containing $j$ we partition the remaining elements in $B(n-j-k)$ ways and color them 2.  This gives us the last term above.
\end{proof}

\subsection{Class 7}

\begin{thm}  For $n\geq1$, we have $$\left|\Pi_n\wr C_2(1^11^2,1^12^2)\right|=(n+1)B(n).$$
\label{T:class37}
\end{thm}

\begin{proof}  Suppose $\sigma\in\Pi_n\wr C_2(1^11^2,1^12^2)$, and suppose $i$ is the smallest element colored 1.  Every element larger than $i$ must be colored 1 as well.  By assumption everything preceding $i$ must be colored 2.  There are no restrictions on how these elements may be partitioned, so we partition them in $B(n)$ ways.  We can choose the $i$ above in $n+1$ ways, where $i=n+1$ corresponds to the whole partition colored 2.
\end{proof}

The above proof is purely enumerative, but we also provide a bijective proof.  It is known via OEIS entry \seqnum{A052889} that $(n+1)B(n)$ counts the number of elements in $\mathcal{S}_{n+2}(12\!-\!3)[12]$, where $\mathcal{S}_{n+2}(12\!-\!3)[12]$ is the set of permutations avoiding the pattern $12\!-\!3$ and beginning with a copy of $12$.  Notice that these permutation must have $n+2$ as their second element.  The first element may be anything.  The remaining $n$ elements must form a $12\!-\!3$ avoiding permutation.  

\begin{proof}[Alternate Proof of Theorem \ref{T:class37}]
We make use of a bijection of Claesson \cite{C01} that puts $\mathcal{S}_n(1\!-\!23)$ in bijection with $\Pi_n$.  Let $\sigma=B_1/B_2/\dots/B_k\in\Pi_n$.  Order the blocks so that $\min B_1>\min B_2>\cdots> \min B_k$.  Order the elements within the blocks so that the first element in each block is the minimum element in that block and the remaining elements are in decreasing order.  Claesson shows that the map $\tau(\sigma)=B_1B_2\dots B_k$, where the blocks are concatenated is a bijection between $\Pi_n$ and $\mathcal{S}_n(1\!-\!23)$.  

Let $f:\Pi_n\wr C_2(1^11^2,1^12^2)\rightarrow \mathcal{S}_{n+2}(12\!-\!3)[12]$ by $f(\sigma)=q$, where the first element of $q$ is the first element colored 1 in $\sigma$ and the second element of $q$ is $n+2$.  If no element in $\sigma$ is colored $1$ then the first element of $q$ is $n+1$.  Now, we replace the element $i$ in its block by $n+1$, and call this new partition $\hat{\sigma}$.   Then $q=i(n+2)((\tau(\widehat{\sigma}))^c)^r$, where $r$ is the reversal map, $c$ is the complement map and $\tau$ is the map described above.  

The permutation $q$ obtained clearly begins with a copy of $12$.  No copy of $12\!-\!3$ can be formed using the first two elements of $q$ and $((\tau(\widehat{\sigma}))^c)^r$ must also avoid $12\!-\!3$ by above.  Thus, the map is well defined and is a bijection.
\end{proof}

\subsection{Class 8}

\begin{thm}  For $n\geq1$, we have $$\left|\Pi_n\wr C_2(1^11^2,1^21^1)\right|=\sum_{k=0}^n 2^k \begin{Bmatrix}
n\\
k
\end{Bmatrix},$$ where $\begin{Bmatrix}
n\\
k
\end{Bmatrix}$ denotes a Stirling number of the second kind.
\end{thm}

\begin{proof}
If a colored partition avoids $1^11^2$ and $1^21^1$, then all blocks must be monochromatic.  There are no further restrictions.  We may partition $[n]$ into $k$ blocks in $\begin{Bmatrix}
n\\
k
\end{Bmatrix}$ ways, and then color each block monochromatically in $2^k$ ways.  Thus this avoidance sequence is given by $\sum_{k=0}^n 2^k \begin{Bmatrix}
n\\
k
\end{Bmatrix}$.
\end{proof}

\section{Avoiding Three Patterns}\label{S:three}

There are six different Wilf classes consisting of three patterns as seen in Table \ref{T:threedata}.  Again, we consider each class in turn.  Wilf equivalences not addressed by Lemma \ref{L:lemma} are explained in the appropriate subsection below.

\begin{center}
\begin{table}
\begin{center}
\begin{tabular}{|c|c|c|c|}
\hline
Class&Patterns &Sequence&OEIS number\\ \hline
1&$\{1^11^1,1^11^2,1^12^1\}\equiv\{1^11^1,1^21^1,1^12^1\}$&$2,3,0,0,0,0$&Trivial\\
&$\{1^11^1,1^12^1,1^12^2\}\equiv\{1^11^1,1^12^1,1^22^1\}$&&\\
\hline
2&$\{1^11^1,1^12^2,1^22^1\}$&$2,4,2,2,2,\dots$&Trivial\\ \hline
3&$\{1^11^1,1^11^2,1^22^1\}\equiv\{1^11^1,1^21^1,1^12^2\}$&$2,4,6,8,10,12$&\seqnum{A005843}\\
&$\{1^11^2,1^12^1,1^12^2\}\equiv\{1^21^1,1^12^1,1^22^1\}$&&\\
&$\{1^11^2,1^12^1,1^22^1\}\equiv\{1^21^1,1^12^1,1^12^2\}$&&\\\hline
4&$\{1^11^1,1^11^2,1^21^1\}$&$2,4,8,16,32,64$&\seqnum{A000079}\\
&$\{1^11^2,1^21^1,1^12^1\}$&&\\
&$\{1^12^1,1^12^2,1^22^1\}$&&\\ \hline
5&$\{1^11^1,1^11^2,1^12^2\}\equiv\{1^11^1,1^21^1,1^22^1\}$&$2,4,8,17,38,90$&\seqnum{A081124}\\
\hline
6&$\{1^11^2,1^12^2,1^22^1\}\equiv\{1^21^1,1^12^2,1^22^1\}$&$2,5,12,33,108$&\seqnum{A209798}\\\hline
7&$\{1^11^2,1^21^1,1^12^2\}\equiv\{1^11^2,1^21^1,1^22^1\}$&$2,5,14,44,154$&\seqnum{A014322}\\
\hline
\end{tabular}
\end{center}
\caption{Enumeration data for 2-colored partitions avoiding three partition patterns}
\label{T:threedata}
\end{table}
\end{center}

\subsection{Class 1}

\begin{thm}
$\left|\Pi_n \wr C_2(1^11^1,1^11^2,1^12^1)\right|=\left|\Pi_n \wr C_2(1^11^1,1^12^1,1^12^2)\right|=\begin{cases}
2&\text{if } n=1,\\
3&\text{if } n=2,\\
0&otherwise.\\
\end{cases}$
\end{thm}

\begin{proof}
Note that if a partition avoids $1^11^1$ and $1^12^1$, then no two elements may have the same color.  Since we only have 2 colors available, such a partition must have at most 2 elements.

If the partition also avoids $1^11^2$, then if two elements are in the same block, the smaller must have color 2 and the larger must have color 1.  The only partitions that obey this are $1^1$ and $1^2$ in $\Pi_1\wr C_2$, and $1^21^1$, $1^12^2$, and $1^22^1$ in $\Pi_2\wr C_2$.

If the partition also avoids $1^12^2$ then if two elements are in different blocks, the larger element must have color 1 and the smaller element must have color 2.  The only partitions that obey this are $1^1$ and $1^2$ in $\Pi_1\wr C_2$, and $1^11^2$, $1^21^1$, and $1^22^1$ in $\Pi_2\wr C_2$.
\end{proof}

\subsection{Class 2}

\begin{thm}
$\left|\Pi_n\wr C_2(1^11^1,1^12^2,1^22^1)\right|=\begin{cases}
4&\text{if }n=2,\\
2&otherwise.\\
\end{cases}$
\end{thm}

\begin{proof}
Consider a partition in $\Pi_n\wr C_2(1^11^1,1^12^2,1^22^1)$.  One can check the cases where $n=1$ and $n=2$ and readily see that the number of partitions matches the given formula.  If $n\geq 3$, then the only partitions in $\Pi_n\wr C_2(1^11^1,1^12^2,1^22^1)$ are the monochromatic partitions with every element in its own block.  To see this, notice that if more than one element is in a block then the block must have two colors to avoid $1^11^1$.  This block along with an element in any other block will form a copy of $1^12^2$ or $1^22^1$.  If every element is in the same block, we must have a copy of $1^11^1$ since $n\geq3$.  Thus, $\left|\Pi_n\wr C_2(1^11^1,1^12^2,1^22^1)\right|=2$ if $n\not=2$ and $\left|\Pi_2\wr C_2(1^11^1,1^12^2,1^22^1)\right|=4$.
\end{proof}

\subsection{Class 3}

\begin{thm} For $n\geq 1$,
$$\left|\Pi_n\wr C_2(1^11^1,1^11^2,1^22^1)\right|=\left|\Pi_n\wr C_2(1^11^2,1^12^1,1^12^2)\right|=\left|\Pi_n\wr C_2(1^11^2,1^12^1,1^22^1)\right|=2n.$$
\end{thm}

\begin{proof}

Suppose that $\sigma\in \Pi_n\wr C_2(1^11^1,1^11^2,1^22^1)$.  Since $\sigma$ avoids $1^11^1$ it may not have blocks with more that two elements in them.  Suppose that $i$ is the first element colored 2.  The elements $[i-1]$ must all be colored 1, and hence must be in their own blocks.  If one of the remaining $n-i$ elements is colored 1 then it must be in the same block as $i$ and hence must be $i+1$ otherwise a copy of $1^22^1$ is produced.  We have $n$ choices for the first 2 colored element $i$.  In the case when $1\leq i\leq n-1$ we have two choices for $i+1$, either $i+1$ is colored 1 and in the same block as $i$ or $i+1$ is colored 2 and in a different block than $i$.  The rest of the partition and colors are determined.  If $i=n$ we do not have an element $i+1$, but we do need to consider the partition where every element is in its own block and every element is colored 1.  So we will say that the two choices in the case when $i=n$ are for $i$ to have color 1 or 2.  Thus, $|\Pi_n\wr C_2(1^11^1,1^11^2,1^22^1)|=2n$.

We will describe a bijection with the partitions in the set $\Pi_n\wr C_2(1^11^2,1^12^1,1^12^2)$.  The partitions in $\Pi_n\wr C_2(1^11^2,1^12^1,1^12^2)$ may have no more than two blocks.  In the case where there are two blocks every element in one block must be colored 1, and every element in the other block must be colored two, otherwise a copy of $1^12^1$ occurs.  Furthermore, the block containing all of the 2-colored elements must contain the elements $1,2,\dots, i$ for some $i\geq 1$ and the remaining elements are in the 1-colored block.  If we have exactly one block then we may use both colors in this block, but the first $i$ elements must be colored 2 and the remaining $n-i$ elements are colored 1. 

The bijection between $\Pi_n\wr C_2(1^11^1,1^11^2,1^22^1)$ and $\Pi_n\wr C_2(1^11^2,1^12^1,1^12^2)$ will work as follows.  The monochromatic partitions of $[n]$ with $n$ blocks in $\Pi_n\wr C_2(1^11^1,1^11^2,1^22^1)$ are mapped to the monochromatic partions of $[n]$ with one block in $\Pi_n\wr C_2(1^11^2,1^12^1,1^12^2)$.  If $i$ is the first element colored 2, we look at the color of $i+1$.  If $i+1$ is colored 1, we put all of the elements in a single block, and then color everything 2 that was colored 1 and vice versa.  If $i+1$ is colored 2 then we put all of the 1 colored elements in a single block and all of the 2 colored elements in a single block and swap the colors.  The bijection between the sets $\Pi_n\wr C_2(1^11^2,1^12^1,1^12^2)$ and $\Pi_n\wr C_2(1^21^1,1^12^1,1^12^2)$ consists of swapping the colors inside of blocks with more than two colors.  
\end{proof}

\subsection{Class 4}

\begin{thm} For $n\geq 1$,
$$\left|\Pi_n \wr C_2 (1^11^1, 1^11^2, 1^21^1)\right|=\left|\Pi_n \wr C_2 (1^11^2, 1^21^1, 1^12^1)\right|=\left|\Pi_n \wr C_2 (1^12^1, 1^12^2, 1^22^1)\right|=2^n.$$
\end{thm}

\begin{proof}
Consider a partition which avoids $1^11^1$, $1^11^2$, and $1^21^1$.  In this case, no two elements can be in the same block.  Thus, we color each element either 1 or 2 and put each element in its own block.  This gives us the $2^n$ partitions avoiding those three patterns.  

For a partition to avoid $1^11^2$, $1^21^1$, and $1^12^1$, every element in a block must have the same color, and there can be at most two blocks, one colored 1 and the other colored 2.  Thus, we color each element either 1 or 2 and put the 1-colored elements in a single block and the 2-colored elements in a single block.  

For a partition to avoid $1^12^1$, $1^12^2$, and $1^22^1$, we must have every element in the same block.  Thus, we color each element either 1 or 2 and put every element in the same block. 
\end{proof}

\subsection{Class 5}

\begin{thm} For $n\geq 1$,
$\left|\Pi_n\wr C_2(1^11^1,1^11^2,1^12^2)\right|=\sum_{k=0}^n\binom{n}{k}\left\lfloor k/2 \right\rfloor!.$
\end{thm}

\begin{proof}
In order to avoid the last two patterns we observe that a partition must have all of its 2-colored elements appear first followed by its 1-colored elements.  In order to avoid the first pattern we observe that no block can have more than two elements.  We will condition on the number of blocks of size 2. 

Suppose that $\sigma\in \Pi_n\wr C_2(1^11^1,1^11^2,1^12^2)$.  We have that $\sigma$ may have at most $\left\lfloor n/2\right\rfloor$ blocks of size 2.  Suppose that $\sigma$ has $\left\lfloor k/2\right\rfloor$ blocks of size 2.  We choose $k$ elements from the $n$ elements to form these blocks.  If $k$ is even then the smallest $k/2$ of these elements must be colored 2, which implies that all of the elements that are smaller than the largest of these elements must also be colored 2.  Now, the remaining elements must be colored 1.  We match each of the $k/2$ smallest elements with one of the $k/2$ largest elements in $\left\lfloor k/2\right\rfloor!$ ways.  

If $k$ is odd, then we color the partition so that the largest element colored 2 is the $(\left\lfloor k/2 \right\rfloor+1)$st smallest element of the $k$ that we have chosen.  We then match the smallest $\left\lfloor k/2 \right\rfloor$ with the largest $\left\lfloor k/2 \right\rfloor$ of these elements to form blocks of size 2.  This can be done in $\left\lfloor k/2 \right\rfloor!$ ways.  

Summing over $0\leq k\leq n$ gives us the formula above.
\end{proof}

\subsection{Class 6}

\begin{thm}
For $n\geq 1$, $|\Pi_n\wr C_2(1^11^2,1^12^2,1^22^1)|=2B(n)+n-1.$
\end{thm}

\begin{proof}
Consider a partition in $\Pi_n\wr C_2(1^11^2,1^12^2,1^22^1)$.  Clearly, if every element in this partition has the same color then all of these patterns are avoided, so there are $2B(n)$ partitions in $\Pi_n\wr C_2(1^11^2,1^12^2,1^22^1)$ that are monochromatic.  This is the first term above.

Suppose that $j$ is the smallest element colored 1.  To avoid $1^11^2$ and $1^12^2$ every element larger than $j$ must be colored 1.  Also, every element less than $j$ must be colored 2, and to avoid $1^22^1$ every element less than $j$ must be in the same block as $j$.  Notice, that if $j=1$ then the partition is monochromatic, so we ignore this case.  The choice of the smallest element colored 1 determines the colors of the remaining elements.  Furthermore, since we are assuming $j\geq 2$, every element larger than $j$ must be in the same block as $j$, else a copy of $1^22^1$ can be found.  This means that there is exactly one nonmonochromatic partition for each smallest 1-colored element $j$. This gives the $n-1$ term.
\end{proof}

\subsection{Class 7}

\begin{thm} For $n\geq 1$, 
$|\Pi_n\wr C_2(1^11^2,1^12^2,1^21^1)|=\sum_{i=0}^nB(i)B(n-i).$
\end{thm}

\begin{proof}  In order to avoid $1^11^2$ and $1^21^1$ we must have that in any block every element has the same color.  Additionally to avoid $1^12^2$ we must have that all of the blocks with 2-colored elements appear before all of the blocks with 1-colored elements.  

If we choose the first $i$ elements to be 2-colored then there are $B(i)$ ways to partition them, and there are $B(n-i)$ ways to partition the remaining $n-i$ 1-colored elements.  This shows that $|\Pi_n\wr C_2(1^11^2,1^12^2,1^21^1)|=\sum_{i=0}^nB(i)B(n-i)$.
\end{proof}

\section{Avoiding 4 or More Patterns}\label{S:four}

Recall from Section \ref{S:three} that if a set $S$ contains the patterns $1^11^1$ and $1^12^1$ then the enumeration of $\Pi_n\wr C_2(S)$ is trivial.  We will avoid those cases here.  Table \ref{T:fourdata} shows Wilf classes containing 4 patterns.  

\begin{center}
\begin{table}
\begin{center}
\begin{tabular}{|c|c|c|c|}
\hline
Class&Patterns &Sequence&OEIS number\\ \hline
1&$\{1^11^1,1^11^2,1^21^1,1^12^1\}$&$2,2,0,0,0$&Trivial\\
&$\{1^11^1,1^11^2,1^12^1,1^12^2\}\equiv\{1^11^1,1^21^1,1^12^1,1^22^1\}$&&\\
&$\{1^11^1,1^11^2,1^12^1,1^22^1\}\equiv\{1^11^1,1^21^1,1^12^1,1^12^2\}$&&\\
&$\{1^11^1,1^12^1,1^12^2,1^22^1\}$&&\\ \hline
2&$\{1^11^2,1^21^1,1^12^2,1^22^1\}$&$2,2,4,10,30$& Twice \seqnum{A000110}\\
 \hline
3&$\{1^11^1,1^11^2,1^12^2,1^22^1\}\equiv\{1^11^1,1^21^1,1^22^1,1^12^2\}$&$2,3,2,2,2$&Trivial\\\hline
4&$\{1^11^1,1^11^2,1^21^1,1^12^2\}\equiv\{1^11^1,1^11^2,1^21^1,1^22^1\}$&$2,3,4,5,6$&\seqnum{A000027}\\
&$\{1^11^2,1^12^1,1^12^2,1^22^1\}\equiv\{1^21^1,1^12^1,1^12^2,1^22^1\}$&&\\
&$\{1^11^2,1^21^1,1^12^1,1^12^2\}\equiv\{1^11^2,1^21^1,1^12^1,1^22^1\}$&&\\\hline
\end{tabular}
\end{center}
\caption{Enumeration data for 2-colored partitions avoiding four partition patterns}
\label{T:fourdata}
\end{table}
\end{center}

Since Classes 1 and 3 are again trivial, we only provide explanation for Classes 2 and 4 below.  

\subsection{Class 2}

\begin{thm}  For $n\geq 1$, 
$\left|\Pi_n \wr C_2(1^11^2,1^21^1,1^12^2,1^22^1)\right|=2B(n).$
\end{thm}

\begin{proof}
Avoiding $1^11^2$ and $1^21^1$ means that any two elements in the same block must have the same color.  Avoiding $1^12^2$ and $1^22^1$ means that any two elements in different blocks must have the same color.  Thus, these permutations must be monochromatic.  Since there are two colors to choose from there are $2B(n)$ colored partitions of $[n]$ avoiding these patterns.  
\end{proof}

\subsection{Class 4}

We will count the partitions that avoid $1^11^1,$ $1^11^2$, $1^2,1^1$, and $1^12^2$.  The other arguments are similar. 

\begin{thm}  For $n\geq 1$, 
$\left| \Pi_n \wr C_k (1^11^1, 1^11^2, 1^21^1, 1^12^2)\right|=n+1.$
\end{thm}

\begin{proof}
Avoiding $1^11^2$ and $1^21^1$ means that every element in each block has the same color.  Avoiding $1^11^1$ means that every element in a block has a different color.  Thus, our partition must be entirely composed of blocks of size 1.  To avoid $1^12^2$ all of the elements colored 2 must appear before those colored 1.  There are $n+1$ ways to color such a partition.  
\end{proof}

\subsection{Avoiding 5 or 6 patterns}

The Wilf classes containing 5 patterns are mostly trivial, since all but two of them contain both $1^11^1$ and $1^12^1$.  The two interesting classes are $\{1^11^1,1^11^2,1^21^1,1^12^2,1^22^1\}$ and $\{1^12^1,1^11^2,1^21^1,1^12^2,1^22^1\}$.  Comparing to Class 2 in Table \ref{T:fourdata} we see that each partition in these sets must be monochromatic.  The pattern $1^11^1$ forces everything to be in its own block where the pattern $1^12^1$ forces everything to be in separate blocks.  Thus for each $n\geq 1$ there are two elements in $\Pi_n\wr C_2(S)$ where $S$ is one of the two sets above.

Obviously, if the class contains all six distinct patterns of $\Pi_2 \wr C_2$ then when $n\geq2$ there are no partitions avoiding all of these patterns.  

\section{Toward Further Generalization}\label{S:generalizations}

So far we have only considered enumeration of pattern avoiding partitions in $\Pi_n \wr C_2$, that is, where we only use 2 colors on our colored set partitions.  Certainly further generalization is possible, but as the number of colors $k$ increases, exact enumeration becomes much more complicated.  We do, however, present one result for partitions with $k$ colors.

\begin{thm}  $\displaystyle |\Pi_n\wr C_k(1^11^2,1^12^2,1^21^1)|=\sum_{i_1+i_2+\dots+i_k=n}B(i_1)B(i_2)\cdots B(i_k)$.  \end{thm}

\begin{proof}  In order to avoid $1^11^2$ and $1^21^1$ every block may have at most one color.  To avoid $1^12^2$, for colors $c_1<c_2$ we must have that any element colored $c_1$ must appear after every element colored $c_2$.  Let $i_j$ be the number of elements colored $j$.  In this case, we must have that elements $n-i_1+1,\dots,n$ are colored 1, elements $n-i_1-i_2+1,\dots,n-i_1$ are colored 2, etc.  Once we color these elements, we partition the elements having the same colors.  There are $B(i_1)B(i_2)\cdots B(i_k)$ ways to do this.  Summing over all possible colorings gives us the result.  \end{proof}

Throughout this paper, we have considered colored set partitions that avoid 2-colored partitions of $[2]$.  Although we focused primarily on 2-colored set partitions avoiding other partitions in the pattern sense, we were able to completely characterize the appropriate Wilf-equivalence classes and provide bijective relationships with several other combinatorial objects.  Certainly further generalization is possible, particularly by increasing the number of available colors.

\bigskip
\hrule
\bigskip

\noindent 2010 {\it Mathematics Subject Classification}:
Primary 05A18; Secondary 05A15.

\noindent \emph{Keywords: } 
colored set partition, pattern avoidance.

\bigskip
\hrule
\bigskip

\noindent (Concerned with sequences \seqnum{A000027}, \seqnum{A000079}, \seqnum{A000110}, \seqnum{A000898}, \seqnum{A000918}, \seqnum{A001861},\\ \seqnum{A005425}, \seqnum{A005843}, \seqnum{A011965}, \seqnum{A014322}, \seqnum{A052889}, \seqnum{A052944}, \seqnum{A081124}, \seqnum{A208275}, \seqnum{A209629}, \seqnum{A209797}, \seqnum{A209798}, \seqnum{A209801}.)

\bigskip
\hrule
\bigskip

\vspace*{+.1in}
\noindent
Received  2012;
revised version received  2012. 
Published in {\it Journal of Integer Sequences},  2012.

\bigskip
\hrule
\bigskip

\noindent
Return to
\htmladdnormallink{Journal of Integer Sequences home page}{http://www.cs.uwaterloo.ca/journals/JIS/}.
\vskip .1in

\end{document}